\documentclass[a4paper,reqno, 11pt]{amsart}

\usepackage{amsmath}
\usepackage{paralist}
\usepackage{graphics}
\usepackage{epsfig}
\usepackage{amsfonts}
\usepackage{amssymb}
\usepackage{color}
\usepackage{epstopdf}

\newtheorem{theorem}{Theorem}[section]
\newtheorem{lemma}[theorem]{Lemma}

\def\R{{\Bbb R}}

\def\ifl{\iffalse }

\def\bc{\begin{center}}       \def\ec{\end{center}}
\def\ba{\begin{array}}        \def\ea{\end{array}}
\def\be{\begin{equation}}     \def\ee{\end{equation}}
\def\bea{\begin{eqnarray}}    \def\eea{\end{eqnarray}}
\def\beaa{\begin{eqnarray*}}  \def\eeaa{\end{eqnarray*}}

\numberwithin{equation}{section}

\newtheorem{corollary}[theorem]{Corollary}

\numberwithin{equation}{section}

\begin{document}
\author{Hai-Yang Jin}
\address{School of Mathematics, South China University of Technology, Guangzhou 510640, China}
\email{mahyjin@scut.edu.cn}
\author{Tian Xiang}
\address{Institute for Mathematical Sciences, Renmin University of China, Bejing, 100872, China}
\email{txiang@ruc.edu.cn}

\title[Convergence rate for the chemotaxis-haptotaxis ]{Boundedness and  exponential convergence of a chemotaxis model for tumor invasion} 


\begin{abstract}
We revisit the following  chemotaxis system modeling tumor invasion
\begin{equation*}
\begin{cases}
u_t=\Delta u-\nabla \cdot(u\nabla v),& x\in\Omega, t>0,\\
v_t=\Delta v+wz,& x\in\Omega, t>0,\\
w_t=-wz,& x\in\Omega, t>0,\\
z_t=\Delta z-z+u, & x\in\Omega, t>0,\\
\end{cases}
\end{equation*}
in a smooth  bounded domain $\Omega \subset \R^n(n\geq 1)$ with homogeneous Neumann boundary and initial conditions. This model was recently proposed by  Fujie {\it{et al.}} \cite{FIY14} as a model for tumor invasion with the role of extracellular matrix incorporated, and was analyzed by Fujie {\it{et al.}}  \cite{FIWY16}, showing the uniform boundedness and convergence for $n\leq 3$. In this work, we first show that the $L^\infty$-boundedness of the system can be reduced to the boundedness of $\|u(\cdot,t)\|_{L^{\frac{n}{4}+\epsilon}(\Omega)}$ for some $\epsilon>0$ alone, and then, for $n\geq 4$, if the initial data $\|u_0\|_{L^{\frac{n}{4}}}$,  $\|z_0\|_{L^\frac{n}{2}}$ and  $\|\nabla v_0 \|_{L^n}$  are sufficiently small, we are able to  establish the $L^\infty$-boundedness of the system. Furthermore, we show that boundedness implies exponential convergence with explicit convergence rate, which  resolves the open problem left in \cite{FIWY16}. More precisely, it is shown, if $u_0\geq, \not\equiv 0$, then any bounded solution $(u,v,w,z)$ of the tumor invasion model satisfies the following exponential decay estimate:
$$
\ba{ll}
&\Bigr\|(u(\cdot, t),v(\cdot, t),w(\cdot, t), z(\cdot, t))-(\bar{u}_0, \bar{v}_0+\bar{w}_0, 0, \bar{u}_0)\Bigr\|_{L^\infty(\Omega)}\\[0.3cm]
&\leq C\Bigr(e^{-\frac{1}{2}\min\{\lambda_1, \frac{\bar{u}_0}{3}\}t},e^{-\min\{\lambda_1,\frac{\bar{u}_0}{3}\}t},e^{-\frac{\bar{u}_0}{2}t},e^{-\min\{1, \frac{\lambda_1,}{2} \frac{\bar{u}_0}{6}\}t} \Bigr)
\ea
$$
for all $t>0$  and for some constant $C>0$ independent of time $t$.  Here, for a generic function $f$, $\bar{f}$ means the spatial average  of $f$ over $\Omega$ and $\lambda_1(>0)$ is  the first nonzero eigenvalue of $-\Delta$ in $\Omega$ with homogeneous Neumann boundary condition.

\end{abstract}



\maketitle

\numberwithin{equation}{section}
\section{Introduction and main results}
In this paper, we shall  revisit the  chemotaxis system modeling tumor invasion:
\begin{equation}\label{TI-model}
\begin{cases}
u_t=\Delta u-\nabla \cdot(u\nabla v),& x\in\Omega, t>0,\\
v_t=\Delta v+wz,& x\in\Omega, t>0,\\
w_t=-wz,& x\in\Omega, t>0,\\
z_t=\Delta z-z+u, & x\in\Omega, t>0,\\
\frac{\partial u}{\partial \nu}=\frac{\partial v}{\partial \nu}=\frac{\partial z}{\partial \nu}=0, & x\in \partial \Omega, t>0,\\
(u(x,0),  v(x,0), w(x,0), z(x,0))=(u_0(x), v_0(x),  w_0(x), z_0(x)), & x\in \Omega,
\end{cases}
\end{equation}
where  $\Omega \subset  \R^n(n\geq 1)$ a bounded domain with smooth boundary $\partial \Omega$,  $\frac{\partial }{\partial \nu}$ represents the differentiation  with respect to the outward normal vector $\nu$ of $\partial\Omega$ and the prescribed initial data satisfy
\begin{equation}\label{Reg-initial}
(u_0,v_0,w_0,z_0)\in C^0(\bar{\Omega})\times W^{1,\infty}(\Omega)\times C^1(\bar{\Omega})\times  C^0(\bar{\Omega})  \text{ with } u_0,v_0,w_0,z_0\geq 0.
\end{equation}
The system  \eqref{TI-model} was recently proposed  by Fujie et al. \cite{FIY14}  as a modified tumor invasion model with chemotaxis effect of  Chaplain and Anderson   type \cite{CA12}. According to  the cancer phenomena point of view, the unknown functions $u,v,w$ and $z$  denote the molar concentration of tumor cells, {\it{active extracellular  matrix}} (ECM$^*$),   extracellular matrix (ECM)  and matrix degrading enzymes (MDE), respectively.   A particular  core of the model is  to account for a chemoattractant induced by an ECM$^*$, which is produced by a biological reaction between ECM and MDE.  We refer to \cite{FIY14} for more explanations and biological background.

Compared with the {\it{haptataxis}} (in which attractants are non-diffusible) type model usually used to describe cancer invasion as pointed out in \cite{FIWY16}, the cross-diffusion is of {\it{chemotaxis}} type in that it is oriented  toward the higher concentration of the diffusible ECM$^*$, the latter being produced by the static ECM together with the chemical MDE. In contrast to the direct Keller-Segel  prototypical model of chemotaxis process \cite{Ke},
\begin{equation}\label{KS-model}
\begin{cases}
u_t=\Delta u-\nabla \cdot(u\nabla z),& x\in\Omega, t>0,\\
z_t=\Delta z-z+u, & x\in\Omega, t>0,\\
\frac{\partial u}{\partial \nu}=\frac{\partial z}{\partial \nu}=0, & x\in \partial \Omega, t>0,\\
(u(x,0), z(x,0))=(u_0(x),  z_0(x)), & x\in \Omega,
\end{cases}
\end{equation}
 which has  the possibility of  blow-up in a finite/infinite time  depending strongly on the space dimension (No blow-up in 1-D \cite{OY01, HP04}, critical mass blow-up in 2-D \cite{HV97,Jager, Na95, NSY97, SS01, MWpre} and, generic blow-up in $\geq 3$-D \cite{Win10-JDE, Win13}), the indirect chemotaxis mechanism in \eqref{TI-model} has been shown to have a role in enhancing the  regularity  and boundedness properties of solutions \cite{FIWY16}.  Indeed, therein, they showed the  boundedness and uniform convergence of  \eqref{TI-model}  to a certain spatially constant equilibrium for $n\leq 3$. More precisely, they proved:
\begin{itemize}
\item[(R1)]  \textbf{(Boundedness in $\leq 3$-D) } Let  $n\leq 3$ and  \eqref{Reg-initial} hold. Then there exists a uniquely determined quadruple $(u,v,w,z)$ of nonnegative functions defined on $\bar{\Omega}\times [0,\infty)$ which solves \eqref{TI-model} classically and  is bounded in the sense there exists  a constant $C>0$ such that  for all $t>0$
\be\label{bdd-sol.}
\|u(\cdot, t)\|_{L^\infty(\Omega)}+\|v(\cdot, t)\|_{W^{1,\infty}(\Omega)}+\|w(\cdot, t)\|_{L^\infty(\Omega)}+\|z(\cdot, t)\|_{L^\infty(\Omega)}<C.
\ee
\item[(R2)]  \textbf{(Uniform Convergence)} Let  $n\leq 3$ and suppose that \eqref{Reg-initial} hold and $u_0\geq, \not\equiv 0$.  Then the solution of \eqref{TI-model} fulfills
\be\label{convergence-u-z}
\Bigr\|(u(\cdot, t),v(\cdot, t),w(\cdot, t), z(\cdot, t))-(\bar{u}_0, \bar{v}_0+\bar{w}_0, 0, \bar{u}_0)\Bigr\|_{L^\infty(\Omega)}\rightarrow 0
\ee
as $t\rightarrow \infty$, where
\be\label{average-u0-z0}
\bar{u}_0=\frac{1}{|\Omega|}\int_\Omega u_0, \quad \bar{v}_0=\frac{1}{|\Omega|}\int_\Omega v_0, \quad \bar{w}_0=\frac{1}{|\Omega|}\int_\Omega w_0.
\ee
\end{itemize}

The conclusions (R1) and (R2) tell us that, when $n\leq 3$, the  chemotactic  cross-diffusion is substantially overbalanced by random diffusion, and that hence  the overall behavior of the model, with respect to both global solvability and asymptotic behavior, is essentially the same as that  of the correspondingly modified system  without this chemotaxis term. The convergence to the constant equilibrium is uniform as stated in \eqref{convergence-u-z}. While, it is important and curious to ask further questions like: how does the solution converge to that equilibrium, algebraically or exponentially, and, what will happen if $n\geq 4$?  The former indeed is the interesting  open question left in \cite[Remark 4.9]{FIWY16} as to their respective rates of convergence in \eqref{convergence-u-z} except the convergence rate of $w$.  Our primary motivation of this paper is to answer this open question:  by carefully utilizing  the Neumann heat semigroup theory, we show that any bounded solution of \eqref{TI-model} converges not only uniformly but also exponentially to its equilibrium $(\bar{u}_0, \bar{v}_0+\bar{w}_0, 0, \bar{u}_0)$. Moreover, by carefully collecting the appearing constants, we calculate  out their explicit rates of convergence in terms of initial datum $u_0$ and the first nonzero Neumann eigenvalue $\lambda_1$ of $-\Delta$ in $\Omega$.

As for the global boundedness, the result  (R1) above  says that  no  blow-up phenomenon could occur when $n\leq 3$. It is quite natural to ask what will happen when $n\geq 4$, will indirect chemotactic  cross-diffusion still enforce boundedness unconditionally?  To explore this problem, we first observe that the boundedness of the model in the sense of \eqref{bdd-sol.} can be indeed reduced to the spatial $L^{\frac{n}{4}+\epsilon}$-boundedness of its $u$ component alone, and then, under certain smallness conditions on the initial data, we demonstrate that infinite chemotactic aggregation can also be  fully suppressed.  This closes the mathematical completeness of the boundeness for the tumor invasion model \eqref{TI-model}.  Our main results read as follows.
\begin{theorem}[Boundedness and exponential convergence for  \eqref{TI-model}]\label{mthm}
\
\
\
\begin{itemize}
\item[(R3)] (\textbf{Boundedness  criterion  in $n$-D}) Let  $n\geq 1$ and let \eqref{Reg-initial} hold. Then the boundedness of $\|u(\cdot, t)\|_{L^{\frac{n}{4}+\epsilon}(\Omega)}$ for some $\epsilon>0$ alone  is sufficient to guarantee  the boundedness of \eqref{TI-model} in the sense of \eqref{bdd-sol.}.
\item[(R4)] (\textbf{Boundedness  in $>3$-D}) Let $n>3$ and \eqref{Reg-initial} hold.  Then,  one can find $\epsilon_0>0$ such that for $\|u_0\|_{L^{\frac{n}{4}}(\Omega)}<\epsilon_0$, $\|z_0\|_{L^\frac{n}{2}(\Omega)}<\epsilon_0$ and $\|\nabla v_0\|_{L^n(\Omega)}<\epsilon_0$, the tumor invasion model \eqref{TI-model} possesses a unique global-in-time classical solution, which is also bounded according to \eqref{bdd-sol.}.
\item[(R5)] (\textbf{Exponential Convergence}) Let $n\geq 1$ and $u_0\geq,\not\equiv 0$. Then all the bounded solutions of \eqref{TI-model} converge not only uniformly but also exponentially to its equilibrium $(\bar{u}_0, \bar{v}_0+\bar{w}_0, 0, \bar{u}_0)$. More precisely,  for any bounded solution,  there exist $t_0>0$ and $m_i>0$ such that
\be\label{exp-convergence-u-z}
\begin{cases}
\|u(\cdot, t)-\bar{u}_0\|_{L^\infty(\Omega)}\leq m_1e^{-\frac{1}{2}\min\{\lambda_1, \frac{\bar{u}_0}{3}\}t}, & \forall t\geq 0, \\[0.25cm]
\|v(\cdot, t)-(\bar{v}_0+\bar{w}_0)\|_{L^\infty(\Omega)}\leq m_2e^{-\min\{\lambda_1,\frac{\bar{u}_0}{3}\}t)}, & \forall t\geq 0, \\[0.25cm]
\|w(\cdot, t)\|_{L^\infty(\Omega)}\leq m_3e^{-\frac{\bar{u}_0}{2}t}, & \forall t\geq 0,\\[0.25cm]
\|z(\cdot, t)-\bar{u}_0\|_{L^\infty(\Omega)}\leq m_4e^{-\min\{1, \frac{\lambda_1,}{2} \frac{\bar{u}_0}{6}\}t}, & \forall t\geq 0.
\end{cases}
\ee
\end{itemize}
Here, $m_i (i=1,2,3,4)$ are suitably large constants depending on $\lambda_1$, the initial data $u_0, v_0, w_0$ and Sobolev embedding constants, cf. Section 4.
\end{theorem}
Thanks to the free mass conversation that $\|u\|_{L^1(\Omega)}=\|u_0\|_{L^1(\Omega)}$,  the boundedness criterion (R3) simply recovers (R1) obtained in \cite{FIWY16}. This also explains why we always  have unconditional  boundedness  when  $n\leq 3$.  For the direct Keller-Segel model \eqref{KS-model}, it has been known that the $L^{\frac{n}{2}+\epsilon}$-boundedness of $u$ implies its $L^\infty$-boundedness and that certain smallness on the initial data with respect to $L^{\frac{n}{2}}$-norm of $u$ and $L^n$-norm of $\nabla z$ ensures boundedness \cite{Cao15}. Hence, the criterion (R3) and boundedness  (R4) present a mathematical quantification for  the tumor model \eqref{TI-model} as an indirect chemotaxis model compared to the direct chemotaxis model \eqref{KS-model}, in other words, how indirect  the tumor model \eqref{TI-model} is, as a chemotaxis model compared to the direct chemotaxis model \eqref{KS-model}. Besides, we wish to point out that boundedness  (R4)  imposes no restriction on $w_0$.  The exponential convergence (R5) surely sharpens the uniform convergence (R2),  and therefore resolves the open problem left in \cite[Remark 4.9]{FIWY16}. Moreover, for  the tumor model \eqref{TI-model}, our result shows that boundedness implies not only uniform convergence as shown in \cite{FIWY16} but also exponential convergence.

As the direct chemotaxis model \eqref{KS-model} possesses blow-ups in higher dimensions \cite{Win10-JDE, Win13}, it would be interesting to investigate whether or not the indirect chemotactic cross-diffusion could drive blow-up phenomenon  without smallness conditions on the initial data in higher dimensions ($n\geq 4$)? While, we will leave this for future explorations.

\section{Preliminaries}
In our subsequent  discussions, we need some  well-known smoothing $L^p-L^q$ type estimates for the Neumann heat semigroup in $\Omega$. We list them here for convenience, one can find them in \cite[Lemma 1.3]{Win10-JDE} and \cite[Lemma 2.1]{Cao15}.
\begin{lemma}\label{heat group}  Let $(e^{t\Delta})_{t>0}$ be the Neumann heat semigroup in $\Omega$ and let $\lambda_1>0$ be the first nonzero eigenvalue of $-\Delta$ with homgeneous Neumann boundary condition. Then there exist some constants $k_i(i=1,2,3,4)$ depending only on $\Omega$ such that
\begin{itemize}
\item[(i)] If $1\leq q\leq p\leq \infty$, then
\begin{equation}\label{Lp-3}
\|e^{t\Delta}f\|_{L^p(\Omega)}\leq k_1\left(1+t^{-\frac{n}{2}(\frac{1}{q}-\frac{1}{p})}\right)e^{-\lambda_1 t}\|f\|_{L^q(\Omega)}  \text{  for all  } t>0
\end{equation}
holds for all $f\in L^q(\Omega)$ satisfying $\int_\Omega f=0$.
\item[(ii)]  If  $1\leq q\leq p\leq\infty$, then
\begin{equation}\label{Lp-2}
\|\nabla e^{t\Delta}f\|_{L^p(\Omega)}\leq k_2\left(1+t^{-\frac{1}{2}-\frac{n}{2}(\frac{1}{q}-\frac{1}{p})}\right)e^{-\lambda_1 t}\|f\|_{L^q(\Omega)}  \text{  for all  } t>0
\end{equation}
is valid for all $f\in L^q(\Omega)$.
\item[(iii)] If $2\leq q\leq p<\infty$, then
\begin{equation}\label{Lp-1}
\|\nabla e^{t\Delta}f\|_{L^p(\Omega)}\leq k_3\Bigr(1+t^{-\frac{n}{2}(\frac{1}{q}-\frac{1}{p}}\Bigr)e^{-\lambda_1t}\|\nabla f\|_{L^q(\Omega)} \text{  for all  } t>0
\end{equation}
is true for all $f\in W^{1,p}(\Omega)$.

\item[(iv)] If $1< q\leq p\leq \infty$, then
\begin{equation}\label{Lp-4}
\|e^{t\Delta}\nabla\cdot f\|_{L^p(\Omega)}\leq k_4\left(1+t^{-\frac{1}{2}-\frac{n}{2}(\frac{1}{q}-\frac{1}{p})}\right)e^{-\lambda_1 t}\|f\|_{L^q(\Omega)}  \text{  for all  } t>0
\end{equation}
is valid for all $f\in (W^{1,p}(\Omega))^n$.
\end{itemize}
\end{lemma}
For convenience, we also write down the local well-posedness of \eqref{TI-model} and its  extendibility criterion from \cite[Theorem 3.1]{FIY14} or \cite[Lemma 2.1]{FIWY16}.
\begin{lemma}\label{local-in-time}Let  $n\geq 1$,   $\Omega\subset \mathbb{R}^n$ be a bounded domain with a smooth boundary and let the initial data $(u_0,v_0, w_0,z_0)$ satisfy \eqref{Reg-initial}.  Then there exist $T_m\in(0, \infty]$ and a unique,  nonnegative,  classical  solution $(u,v, w, z)$ of   (\ref{TI-model}) on $\Omega\times [0, T_m)$ such that
$$\ba{ll}
u\in C(\bar{\Omega}\times [0, T_m))\cap C^{2,1}(\bar{\Omega}\times (0, T_m)), \\[0.2cm]
v\in C(\bar{\Omega}\times [0, T_m))\cap C^{2,1}(\bar{\Omega}\times (0, T_m))\cap L_{\text{loc}}^\infty([0, T_m); W^{1,\infty}(\Omega))\\[0.2cm]
w\in C(\bar{\Omega}\times [0, T_m))\cap C^{0,1}(\bar{\Omega}\times (0, T_m)), \\[0.2cm]
z\in C(\bar{\Omega}\times [0, T_m))\cap C^{2,1}(\bar{\Omega}\times (0, T_m)).
\ea $$
Furthermore,   if $T_m<\infty$, then
\be\label{T_m-cri} \|u(\cdot, t)\|_{L^\infty(\Omega)}+\|v(\cdot, t)\|_{W^{1,\infty}(\Omega)}+\|z(\cdot, t)\|_{L^\infty(\Omega)}\rightarrow \infty \quad \quad \mbox{ as }  t\rightarrow T_m^{-}.
\ee
\end{lemma}
The regularity  \eqref{Reg-initial} will be assumed in force henceforth. The following properties are simple observations from the equations in \eqref{TI-model} (see \cite[Lemmas 2.2, 2.3 and 2.4] {FIWY16}).
\begin{lemma}\label{basic prop}The solution $(u,v,w,z)$ of \eqref{TI-model} enjoys the following properties.
\begin{itemize}
\item[(i)] The $u$-component has the mass conservation:
$$
\int_\Omega u(x,t)dx=\int_\Omega u_0(x)dx, \quad \forall t\in(0, T_m).
$$
\item[(ii)]The $v$- and $w$-components have  the mass conservation that
$$
\int_\Omega v(x,t)dx+\int_\Omega w(x,t)dx=\int_\Omega v_0(x)dx+\int_\Omega w_0(x)dx, \quad \forall t\in(0, T_m).
$$
\item[(iii)] The $w$-component is bounded in the way that
$$
\|w(\cdot, t)\|_{L^\infty(\Omega)}\leq \|w_0\|_{L^\infty(\Omega)}, \quad \forall t\in(0, T_m).
$$
\end{itemize}
\end{lemma}
In the subsequent  text, we shall denote by  $k_i (i=1,2,3,4)$ the constants appearing in Lemma \ref{heat group}  and by other  $c_i$ or $C$ or  $C_{\cdot}$ generic constants, which may vary from line to line. In most places, we shall write the commonly used short notations like
$$
\|f(\cdot,t)\|_{L^p}=\|f(\cdot,t)\|_{L^p(\Omega)}=\Bigr(\int_\Omega |f(x,t)|^pdx\Big)^{\frac{1}{p}}.
$$

\section{Boundedness in four and  higher dimensions}
In this section, our goal is to extend the boundedness in  $\leq 3$-dimensions of \cite{FIWY16} to in $\geq 4$-dimensions under some smallness condition on the initial data. Before proceeding, we first show a boundedness principle which reduces the hard task of proving the $L^\infty$- boundeness of \eqref{TI-model}   to the less hard task of the $L^{\frac{n}{4}+\epsilon}$-boundedness of its $u$ component.
\begin{theorem}\label{bddness-cond} If there  exist  some $\epsilon>0$  and $M>0$   such that
$$
\|u(\cdot, t)\|_{L^{\frac{n}{4}+\epsilon}(\Omega)}\leq M, \quad \forall t\in (0, T_m),
$$
then $T_m=\infty$ and the solution $(u,v,w,z)$ of \eqref{TI-model}  is bounded in the sense of \eqref{bdd-sol.}.  Moreover, there exist some $\alpha\in (0, 1)$  and constant $C=C(n, \epsilon, M)$ such that
\be\label{c-alpha-est}
\|u\|_{C^{2+\alpha, 1+\frac{\alpha}{2}}(\bar{\Omega}\times[t, t+1])}+\|v\|_{C^{2+\alpha, 1+\frac{\alpha}{2}}(\bar{\Omega}\times[t, t+1])}+\|z\|_{C^{2+\alpha, 1+\frac{\alpha}{2}}(\bar{\Omega}\times[t, t+1])}\leq C, \quad  \forall t\geq 1.
\ee
\end{theorem}
\begin{proof} This is proven via variation-of-constants-formulas from semigroup representation. Its proof will  repeatedly utilize  \cite[Lemmas 2.3, 3.1, 3.2 and 3.4]{FIWY16}. To start off,  in view of the mass conservation   $\|u(\cdot, t)\|_{L^1}=\|u_0\|_{L^1}$, we can take $\epsilon=1-\frac{n}{4}$ for $n<4$. Hence, we shall proceed with   $p=\frac{n}{4}+\epsilon\geq 1$,  and then  from  \cite[Lemma 3.1]{FIWY16} we infer that
\be\label{z-lq}
\|z(\cdot, t)\|_{L^q}\leq C_z(n, \epsilon,q, M), \quad \forall t\in (0, T_m)
\ee
holds for all
\be\label{q-exp}
1\leq q<\frac{np}{n-2p} \text{  if  } p\leq \frac{n}{2}, \quad 1\leq q\leq \infty \text{  if  } p>\frac{n}{2}.
\ee
In the case of $p>\frac{n}{2}$, we take $q=\infty$ in \eqref{z-lq}, and then we see easily from \cite[Lemmas 3.2 and 3.3]{FIWY16} that $\|\nabla v(\cdot, t)\|_{L^\infty}$ and $\|u(\cdot, t)\|_{L^\infty}$ are uniformly bounded on $(0, T_m)$. Then we are simply done, cf, the end proof of this theorem.

Thus, we will continue with $\frac{n}{4}<p\leq\frac{n}{2}$, which then  entails $\frac{np}{n-2p}>\frac{n}{2}$. By H\"{o}lder inequality, we can take $q\in (\frac{n}{2}, n)$ in \eqref{z-lq}.  Hence, by the semigroup action on the $v$-equation \cite[Lemma 3.2]{FIWY16}, it follows that
\be\label{gradv-lr}
\|\nabla v(\cdot, t)\|_{L^r}\leq  C_v(n, \epsilon,r, M), \quad \forall t\in (0, T_m)
\ee
holds for all
\be\label{r-exp}
1\leq r<\frac{nq}{n-q}.
\ee
The choice of $q>\frac{n}{2}$ directly gives $\frac{nq}{n-q}>n$. Hence, one can choose $r>n$ in \eqref{gradv-lr} and use \cite[Lemma 3.4]{FIWY16} to see that $u$ is uniformly bounded:
\be\label{u-linfty}
\|u(\cdot, t)\|_{L^\infty}\leq  C_u(n, \epsilon,r, q, M), \quad \forall t\in (0, T_m).
\ee
Here, we would like to supply a short proof for \eqref{u-linfty}. Indeed, by the variation-of-constants formula for $u$, we have from (iv) of Lemma \ref{heat group} that
\be\label{u-linfty-est}\ba{ll}
&\|u(\cdot, t)\|_{L^\infty}\\[0.25cm]
&\leq \|e^{t\Delta} u_0\|_{L^\infty}+\int_0^t \|e^{(t-s)\Delta} \nabla\cdot(u(\cdot, s)\nabla v(\cdot, s))\|_{L^\infty}ds\\[0.25cm]
&\leq  \|u_0\|_{L^\infty}+k_4\int_0^t (1+(t-s)^{-\frac{1}{2}-\frac{n}{2\theta}})e^{-\lambda_1 (t-s)}\|u(\cdot, s)\nabla v(\cdot, s)\|_{L^\theta}ds.
\ea
\ee
By taking $\theta\in (\frac{r}{r+1}, r)$, we then infer from H\"{o}lder inequality,  interpolation inequality, \eqref{gradv-lr} and  the mass conservation of $u$, c.f. (i) of Lemma \ref{basic prop}  that
\be\label{ugrav -est}
\ba{ll}
\|u(\cdot, s)\nabla v(\cdot, s)\|_{L^\theta}&\leq \|u(\cdot, s)\|_{L^{\frac{r\theta}{r-\theta}}}\|\nabla v(\cdot, s)\|_{L^r}\\[0.25cm]
&\leq \|u(\cdot, s)\|_{L^\infty}^{1-\frac{r-\theta}{r\theta}}\|u(\cdot, s)\|_{L^1}^{\frac{r-\theta}{r\theta}}\|\nabla v(\cdot, s)\|_{L^r}\\[0.2cm]
&\leq  C_v(n, \epsilon,r, M)\|u_0\|_{L^1}^{\frac{r-\theta}{r\theta}} \|u(\cdot, s)\|_{L^\infty}^{1-\frac{r-\theta}{r\theta}}.
\ea
\ee
For any $t\in (0, T_m)$, set $S(t):= \sup_{s\in (0, t)}\|u(\cdot, s)\|_{L^\infty}$. Then, since $S(t)$ is nondecreasing for $t\in (0, T_m)$, we conclude from \eqref{u-linfty-est} and \eqref{ugrav -est} that
$$
S(t)
\leq  \|u_0\|_{L^\infty}+c_1S^{\frac{r(\theta-1)+\theta}{r\theta}}(t)\int_0^t (1+(t-s)^{-\frac{1}{2}-\frac{n}{2\theta}})e^{-\lambda_1 (t-s)}ds
$$
with $c_1=k_4C_v(n, \epsilon,q, M)\|u_0\|_{L^1}^{\frac{r-\theta}{r\theta}}$.  It follows from  $\theta>n$ the integral   $\int_0^t (1+(t-s)^{-\frac{1}{2}-\frac{n}{2\theta}})e^{-\lambda_1 (t-s)}ds:=c_2<\infty$. Therefore, by Young's inequality, for any $t\in (0, T_m)$,
\begin{equation*}
\begin{split}
S(t)
\leq  \|u_0\|_{L^\infty}+c_1c_2S(t)^{\frac{r(\theta-1)+\theta}{r\theta}}\leq &\|u_0\|_{L^\infty}+\frac{1}{2}S(t)\\
 &+\frac{r-\theta}{r\theta}\Bigr[\Bigr(\frac{2r\theta}{(r+1)\theta-r}\Bigr)^{\frac{r\theta}{(r+1)\theta-r}}c_1c_2\Bigr]^{\frac{r\theta}{r-\theta}},
 \end{split}
 \end{equation*}
which implies the uniform boundedness of $u$:
$$
\|u(\cdot, t)\|_{L^\infty}\leq S(t)\leq 2\Big\{ \|u_0\|_{L^\infty}+\frac{r-\theta}{r\theta}\Bigr[\Bigr(\frac{2r\theta}{(r+1)\theta-r}\Bigr)^{\frac{r\theta}{(r+1)\theta-r}}c_1c_2\Bigr]^{\frac{r\theta}{r-\theta}}\Bigr\},\ \forall t\in (0, T_m).
$$
 Now, it is fairly easy to show the boundedness of $\|v(\cdot, t)\|_{W^{1,\infty}}$ and $\|z(\cdot, t)\|_{L^\infty}$. Twice applications of the maximum principle  to  the  fourth  and then second equation in \eqref{TI-model} deduce $\|z(\cdot, t)\|_{L^\infty}\leq\|u(\cdot, t)\|_{L^\infty}$ and $\|v(\cdot, t)\|_{L^\infty}\leq\|w_0\|_{L^\infty}\|z(\cdot, t)\|_{L^\infty}$,  respectively.    Hence, by the semigroup action on the $v$-equation, it follows by taking $q>n$ in  \cite[Lemma 3.2]{FIWY16}  that $\|\nabla v(\cdot, t)\|_{L^\infty}$ is uniformly bounded on $(0, T_m)$. These together with the extendibility criterion \eqref{T_m-cri} forces $T_m=\infty$. The H\"{o}lder regularity \eqref{c-alpha-est} follows from standard bootstrap arguments involving interior parabolic regularity theory \cite{La}.
\end{proof}
By the mass conversation of $u$ in (i) of Lemma \ref{basic prop}, a direct application of Theorem \ref{bddness-cond} recovers the unconditional boundedness for $n\leq 3$ in \cite{FIWY16}.
\begin{corollary}For $n\leq 3$, any  solution $(u,v,w,z)$ of the tumor invasion model \eqref{TI-model}  is bounded in the sense of \eqref{bdd-sol.} and \eqref{c-alpha-est}.
\end{corollary}
Next, with the boundedness criterion at hand, we shall show that under some smallness assumptions on the initial data, the classical solution will exist globally  with uniform-in-time bound in the case $n\geq 4$.
\begin{theorem}\label{BN4}
Let $n\geq 4$ and the initial data $(u_0,v_0, w_0,z_0)$ satisfy \eqref{Reg-initial}. If there exists $\varepsilon_0>0$  such that
\begin{equation}\label{Initial conditions}
\|u_0\|_{L^\frac{n}{4}(\Omega)}\leq \varepsilon\ ,\ \ \|z_0\|_{L^\frac{n}{2}(\Omega)}\leq \varepsilon\ \ \mathrm{and}\ \|\nabla v_0\|_{L^n(\Omega)}\leq \varepsilon
\end{equation}
for some $\varepsilon<\varepsilon_0$, then the solution $(u,v,w,z)$ of the tumor invasion model \eqref{TI-model} exists globally in time and is bounded according to \eqref{bdd-sol.}.
\end{theorem}
\begin{proof} Solving  the third  equation in \eqref{TI-model} trivially shows
$$
w(x,t)=w_0(x)e^{-\int_{0}^t z(x,s)ds}.
$$
Using the similar argument as in \cite[Lemma 4.9]{FIWY16} and \cite[Lemma 2.9]{TW15-SIAM}, cf. also the discussion in Lemma \ref{decay-w} below, we can find a constant $\delta_1>0$ such that $z(t,x)>\delta_1$ for all $x\in\Omega$ and $t>1$. Hence,
\begin{equation}\label{convergence of w}
\|w\|_{L^\infty}\leq \|w_0\|_{L^\infty} e^{-\delta_1 t}.
\end{equation}
Let $0<\kappa<\min\{\delta_1,\lambda_1\}$. With $\varepsilon_0>0$ to be specified below  and assume that \eqref{Initial conditions} hold for $\varepsilon\in(0,\varepsilon_0)$.  We define
\begin{equation}\label{Definition of T}
\begin{split}
T:=\sup\Bigr\{\tilde{T}>0\Big|&\|u(\cdot,t)-e^{t\Delta }u_0\|_{L^{\theta}}\leq\varepsilon (1+t^{-2+\frac{n}{2\theta}}) e^{-\kappa t}\\
 &\mathrm{for \ all}\  t\in[0,\tilde{T})\ \mathrm{and} \ \ \mathrm{for \ all\ } \theta\in(\frac{n}{4},\frac{n}{3})\Bigr\}.
 \end{split}
\end{equation}
 It is easy to  check that $T$ is well defined and positive due to Lemma \ref{local-in-time}. In the sequel, we need to show that $T=\infty$. First,  since  $n\geq 4$, the smallness condition \eqref{Initial conditions} yields
\begin{equation}\label{bar-u}
\bar{u}_0=\frac{1}{|\Omega|}\int_\Omega u_0 \leq |\Omega|^{-\frac{4}{n}}\|u_0\|_{L^\frac{n}{4}}\leq |\Omega|^{-\frac{4}{n}}\varepsilon.
\end{equation}
Then from the definition of $T$ in \eqref{Definition of T} and \eqref{bar-u}, we deduce from Lemma \ref{heat group} (i)  that
\begin{equation}\label{convergence of u}
\begin{split}
\|u(\cdot,t)-\bar{u}_0\|_{L^{\theta}}
&\leq \|u(\cdot,t)-e^{t\Delta } u_0\|_{L^{\theta}}+\|e^{t\Delta} u_0-\bar{u}_0\|_{L^{\theta}}\\
&\leq \varepsilon (1+t^{-2+\frac{n}{2\theta}}) e^{-\kappa t}+ k_1(1+t^{-2+\frac{n}{2\theta}})e^{-\lambda_1 t}\|u_0-\bar{u}_0\|_{L^\frac{n}{4}} \\
&\leq \varepsilon c_1(1+t^{-2+\frac{n}{2\theta}}) e^{-\kappa t}, \quad \forall t\in (0, T), \quad \theta \in (\frac{n}{4}, \frac{n}{3}).
\end{split}
\end{equation}
From the fourth equation in \eqref{TI-model} and the mass conservation of $u$, it follows  that
\be\label{z-bar}
(\bar{z})_t=-\bar{z}+\bar{u}=-\bar{z}+\bar{u}_0.
\ee
This joins again the fourth equation in \eqref{TI-model} entail
\begin{equation}\label{CZ-1}
(z-\bar{z})_t=\Delta (z-\bar{z})-(z-\bar{z})+u-\bar{u}_0.
\end{equation}
Then the variation-of-constant formula applied to \eqref{CZ-1} shows that
\begin{equation*}
z(\cdot,t)-\bar{z}= e^{t(\Delta-1)}(z_0-\bar{z}_0)+\int_{0}^te^{(t-s)(\Delta-1)}(u(\cdot, s)-\bar{u}_0)ds,
\end{equation*}
which, together with \eqref{Lp-3}, \eqref{Initial conditions} and \eqref{convergence of u},  infers, for $\frac{n}{2}< p<n$, that
\begin{equation}\label{CZ-2}
\begin{split}
&\|z(\cdot,t)-\bar{z}\|_{L^p}\\
&\leq k_1 (1+t^{-1+\frac{n}{2p}}) e^{-(\lambda_1+1)t}\|z_0-\bar{z}_0\|_{L^{\frac{n}{2}}}\\
&\ \ \ +k_1\int_0^t\left(1+(t-s)^{-\frac{n}{2}\left(\frac{1}{\theta}-\frac{1}{p}\right)}\right)e^{-(\lambda_1+1) (t-s)} \|u(\cdot, s)-\bar{u}_0\|_{L^{\theta}}ds\\
&\leq  c_2(1+t^{-1+\frac{n}{2p}})e^{-(\lambda_1+1)t}\|z_0\|_{L^{\frac{n}{2}}} \\
&\ \ \ +k_1c_1\varepsilon\int_0^t\left(1+(t-s)^{-\frac{n}{2}\left(\frac{1}{\theta}-\frac{1}{p}\right)}\right)(1+s^{-2+\frac{n}{2\theta}})e^{-(\lambda_1+1) (t-s)} e^{-\kappa s}ds\\
&\leq c_2\varepsilon(1+t^{-1+\frac{n}{2p}})e^{-(\lambda_1+1)t}+k_1c_1C\varepsilon(1+t^{\min\{0,-1+\frac{n}{2p}\}})e^{-\kappa t}\\
&\leq c_3\varepsilon(1+t^{-1+\frac{n}{2p}})e^{-\kappa t},
\end{split}
\end{equation}
where we have chosen  $\theta$ such that  $\frac{1}{\theta}=\frac{1}{p}-\delta_2+\frac{2}{n}$  with $\delta_2\in (0, \frac{1}{p}-\frac{1}{n})$. In this way, since $p\in (\frac{n}{2}, n)$ we are ensured that  $-\frac{n}{2}(\frac{1}{\theta}-\frac{1}{p})>-1$ as well as $\theta\in (\frac{n}{4}, \frac{n}{3})$. The latter further gives $-2+\frac{n}{2\theta}>-1$. These in conjunction with  $\kappa<\lambda_1$ enable us  to employ  the  following algebraic result  \cite[Lemma 1.2]{Win10-JDE}:
\begin{lemma}\label{UL}
 Let $\alpha<1$, $\beta<1$ and $\gamma,\delta$ be positive constants satisfying $\gamma\neq\delta$. Then there exists some positive constant $C$ depending on $\alpha,\beta,\delta,\gamma$ such that, for all $t>0$,
\begin{equation*}
\int_0^t(1+(t-s)^{-\alpha})e^{-\gamma(t-s)}(1+s^{-\beta})e^{-\delta s}ds\leq C(1+t^{\min\{0,1-\alpha-\beta\}})e^{-\min\{\gamma,\delta\}t}.
\end{equation*}
\end{lemma}
In the next section,  at some places, we even need a refined version of such type result and will compute out  the constant $C$ explicitly.

By the semigroup representation of the second equation in \eqref{TI-model} and Lemma \ref{heat group} (ii), for all $q\in (n,\infty)$,  we deduce that
\begin{equation}\label{CV-1}
\begin{split}
&\|\nabla (v(\cdot,t)-e^{t\Delta}v_0)\|_{L^q}\\
&\leq \int_0^t\| \nabla e^{(t-s)\Delta} w(z-\bar{z})\|_{L^q}ds+\int_0^t\| \nabla e^{(t-s)\Delta} w\bar{z}\|_{L^q}ds\\
&\leq k_2\int_0^t\left(1+(t-s)^{-\frac{1}{2}-\frac{n}{2}\left(\frac{1}{p}-\frac{1}{q}\right)}\right)e^{-\lambda_1 (t-s)}\|w\|_{L^\infty}\|z-\bar{z}\|_{L^p}ds\\
&\ \ \ \ + k_2\int_0^t\left(1+(t-s)^{-\frac{1}{2}}\right)e^{-\lambda_1 (t-s)}\|w\|_{L^\infty}\|\bar{z}\|_{L^q}ds\\
&:=J_1+J_2.
\end{split}
\end{equation}
For any $q> n$, we first fix  $\delta_3\in (0, \frac{1}{q})$ and then take  $p$ satisfying $\frac{1}{p}=\frac{1}{n}+\frac{1}{q}-\delta_3$ so that $-\frac{1}{2}-\frac{n}{2}(\frac{1}{p}-\frac{1}{q})>-1$ and $\frac{n}{2}<p<n$. Therefore, combining  \eqref{convergence of w} and \eqref{CZ-2} and applying Lemma \ref{UL}, we conclude that
\begin{equation}\label{J1}
\begin{split}
J_1&\leq\varepsilon c_3 k_2 \|w_0\|_{L^\infty}\int_0^t\left(1+(t-s)^{-\frac{1}{2}-\frac{n}{2}\left(\frac{1}{p}-\frac{1}{q}\right)}\right)(1+s^{-1+\frac{n}{2p}})e^{-\lambda_1 (t-s)}e^{-\kappa s}ds\\
&\leq  \varepsilon c_4 (1+t^{\min\{0,-\frac{1}{2}+\frac{n}{2q}\}})e^{-\kappa t}\\
&=\varepsilon c_4(1+t^{-\frac{1}{2}+\frac{n}{2q}})e^{-\kappa t}.
\end{split}
\end{equation}
To estimate $J_2$, we first solve \eqref{z-bar} to obtain $\bar{z}=\bar{z}_0 e^{-t}+(1-e^{-t})\bar{u}_0$, and then
\begin{equation}\label{barz}
\begin{split}
\|\bar{z}\|_{L^q}
&\leq \|\bar{z}_0e^{-t}\|_{L^q}+\|(1-e^{-t})\bar{u}_0\|_{L^q}\\
&=e^{-t}|\Omega|^{-1+\frac{1}{q}}\|z_0\|_{L^1}+(1-e^{-t})|\Omega|^{-1+\frac{1}{q}}\|u_0\|_{L^1}\\
&\leq e^{-t}|\Omega|^{\frac{1}{q}-\frac{2}{n}}\|z_0\|_{L^\frac{n}{2}}+(1-e^{-t})|\Omega|^{\frac{1}{q}-\frac{4}{n}}\|u_0\|_{L^\frac{n}{4}}\\
&\leq (|\Omega|^{\frac{1}{q}-\frac{2}{n}}+|\Omega|^{\frac{1}{q}-\frac{4}{n}})\varepsilon\leq 2\max\{1, |\Omega|^{-\frac{1}{n}}, |\Omega|^{-\frac{3}{n}}\}\varepsilon.
\end{split}
\end{equation}
Using \eqref{convergence of w}, \eqref{barz},  Lemma \ref{UL} and noting that $0<\kappa<\min\{\lambda_1,\delta_1\}$, we have
\begin{equation}\label{J2}
\begin{split}
J_2
&\leq 2\varepsilon k_2\|w_0\|_{L^\infty} \max\{1, |\Omega|^{-\frac{1}{n}}, |\Omega|^{-\frac{3}{n}}\} \int_0^t \left(1+(t-s)^{-\frac{1}{2}}\right) e^{-\lambda_1 (t-s)}e^{-\delta_1 s}ds\\
&\leq \varepsilon c_5 e^{-\min\{\lambda_1,\delta_1\}t}\\
&\leq \varepsilon c_5(1+t^{-\frac{1}{2}+\frac{n}{2q}})e^{-\kappa t}.
\end{split}
\end{equation}
A substitution of \eqref{J1} and \eqref{J2} into \eqref{CV-1} gives rise to
\begin{equation}\label{V-v0}
\|\nabla (v(\cdot,t)-e^{\Delta t}v_0)\|_{L^q}\leq \varepsilon c_{6} (1+t^{-\frac{1}{2}+\frac{n}{2q}})e^{-\kappa t}.
\end{equation}
For $q>n\geq 4$, a simple application of \eqref{Initial conditions} and \eqref{Lp-1} yields
\begin{equation}\label{v0-2}
\begin{split}
\|\nabla e^{\Delta t}v_0\|_{L^q}\leq k_3 (1+t^{-\frac{1}{2}+\frac{n}{2q}})e^{-\lambda_1 t}\|\nabla v_0\|_{L^n} \leq \varepsilon c_7 (1+t^{-\frac{1}{2}+\frac{n}{2q}})e^{-\lambda_1 t}.
\end{split}
\end{equation}
Joining \eqref{V-v0} and   \eqref{v0-2}, we arrive at
\begin{equation}\label{VK}
\begin{split}
\|\nabla  v(\cdot,t)\|_{L^q}
&\leq \|\nabla e^{\Delta t} v_0\|_{L^q}+\|\nabla (v(\cdot,t)-e^{\Delta t}v_0)\|_{L^q}\\
&\leq \varepsilon c_{7}(1+t^{-\frac{1}{2}+\frac{n}{2q}})e^{-\lambda_1 t}+\varepsilon c_{6} (1+t^{-\frac{1}{2}+\frac{n}{2q}})e^{-\kappa t}\\
&\leq\varepsilon c_{8} (1+t^{-\frac{1}{2}+\frac{n}{2q}})e^{-\kappa t}, \quad \forall t\in (0, T), \quad q\in (n, \infty).
\end{split}
\end{equation}
Now, we apply variation-of-constant formula to  the first equation in \eqref{TI-model} to obtain
\begin{equation}\label{CU-1}
u(\cdot,t)=e^{t\Delta} u_0-\int_0^t e^{(t-s)\Delta}\nabla \cdot(u(\cdot,s)\nabla v(\cdot,s))ds.
\end{equation}
In light of Lemma \ref{heat group} (iv), we infer from \eqref{CU-1}  that
\begin{equation}\label{CU-2}
\begin{split}
&\|u(\cdot,t)-e^{t\Delta} u_0\|_{L^\theta}\\
&\leq k_4\int_0^t\left(1+(t-s)^{-\frac{1}{2}}\right)e^{-\lambda_1 (t-s)}\|u(\cdot,s)\nabla v(\cdot,s)\|_{L^{\theta}}ds\\
&\leq k_4\int_0^t\left(1+(t-s)^{-\frac{1}{2}}\right)e^{-\lambda_1 (t-s)}\|(u(\cdot,s)-\bar{u}_0)\nabla v(\cdot,s)\|_{L^{\theta}}ds\\
&\ \ \ \ +k_4\int_0^t\left(1+(t-s)^{-\frac{1}{2}}\right)e^{-\lambda_1 (t-s)}\|\bar{u}_0\nabla v(\cdot,s)\|_{L^{\theta}}ds\\
&:=I_1+I_2.
\end{split}
\end{equation}
Next, we  will estimate $I_1$ and $I_2$. As for $I_1$, for any $\theta\in(\frac{n}{4},\frac{n}{3})$ and  $ \theta_1\in(\theta,\frac{n}{3})$, we set
 $$
 \theta_2:=\theta_2(\theta,\theta_1)=\frac{\theta\theta_1}{\theta_1-\theta}\Longleftrightarrow\frac{1}{\theta_1}+\frac{1}{\theta_2}=\frac{1}{\theta}.
 $$
  Then it is a little tedious but easy to verify that $\frac{n}{4}<\theta_1<\frac{n}{3}$, $\theta_2>\theta_2( \frac{n}{4},\frac{n}{3})=n$,  $-2+\frac{n}{2\theta_1}<0$, $-\frac{1}{2}+\frac{n}{2\theta_2}<0$ and  $-\frac{5}{2}+\frac{n}{2\theta}>-1$. Then from  \eqref{convergence of u}, \eqref{VK},  Lemma $\ref{UL}$ and H\"{o}lder's inequality, we infer that
\begin{equation}\label{I1}
\begin{split}
I_1&\leq k_4\int_0^t\left(1+(t-s)^{-\frac{1}{2}}\right)e^{-\lambda_1 (t-s)}\|u-\bar{u}_0\|_{L^{\theta_1}}\|\nabla v\|_{L^{\theta_2}}ds\\
&\leq c_{9}\varepsilon^2 \int_0^t\left(1+(t-s)^{-\frac{1}{2}}\right)e^{-\lambda_1 (t-s)}(1+s^{-2+\frac{n}{2\theta_1}})e^{-\kappa s}\cdot (1+s^{-\frac{1}{2}+\frac{n}{2\theta_2}})e^{-\kappa s}ds\\
&\leq c_{10}\varepsilon^2\int_0^t\left(1+(t-s)^{-\frac{1}{2}}\right)e^{-\lambda_1 (t-s)}(1+s^{-\frac{5}{2}+\frac{n}{2\theta}})e^{-\kappa s}ds\\
& \leq c_{11}\varepsilon^2(1+t^{\min\{0,-2+\frac{n}{2\theta}\}})e^{-\kappa t}\\
&= c_{11}\varepsilon^2(1+t^{-2+\frac{n}{2\theta}})e^{-\kappa t}.
\end{split}
\end{equation}
From \eqref{bar-u} and \eqref{VK}, Lemma $\ref{UL}$ and H\"{o}lder's inequality,  we control  $I_2$ as follows:
\begin{equation}\label{I2}
\begin{split}
I_2&\leq k_4\int_0^t\left(1+(t-s)^{-\frac{1}{2}}\right)e^{-\lambda_1 (t-s)}\|\bar{u}_0\|_{L^{\theta_1}}\|\nabla v(\cdot,s)\|_{L^{\theta_2}}ds\\
& \leq k_4c_8|\Omega|^{\frac{1}{\theta_1}-\frac{4}{n}}\varepsilon^2\int_0^t\left(1+(t-s)^{-\frac{1}{2}}\right)e^{-\lambda_1 (t-s)}(1+s^{-\frac{1}{2}+\frac{n}{2\theta_2}})e^{-\kappa s}ds\\
&\leq k_4c_8 \max\{|\Omega|^{-\frac{1}{n}}, 1\}\varepsilon^2\int_0^t\left(1+(t-s)^{-\frac{1}{2}}\right)e^{-\lambda_1 (t-s)}(1+s^{-2+\frac{n}{2\theta_1}})\\
&\ \ \ \ \ \ \ \ \ \ \ \ \quad \quad \quad \quad \quad \quad \quad \times (1+s^{-\frac{1}{2}+\frac{n}{2\theta_2}})e^{-\kappa s}ds\\
& \leq  c_{12}\varepsilon^2 \int_0^t\left(1+(t-s)^{-\frac{1}{2}}\right)e^{-\lambda_1 (t-s)}(1+s^{-\frac{5}{2}+\frac{n}{2\theta}})e^{-\kappa s}ds\\
&\leq c_{13}\varepsilon^2(1+t^{-2+\frac{n}{2\theta}})e^{-\kappa t}.
\end{split}
\end{equation}
Substituting \eqref{I1} and \eqref{I2} into \eqref{CU-2}, we finally obtain the crucial estimate for $u$:
\begin{equation}\label{CU-3}
\|u(\cdot,t)-e^{t\Delta} u_0\|_{L^\theta}\leq c_{14}\varepsilon^2(1+t^{-2+\frac{n}{2\theta}})e^{-\kappa t}.
\end{equation}
Hence, upon choosing $\varepsilon_0<\frac{1}{2c_{14}}$ in \eqref{CU-3}, we argue by contraction from the definition of $T$ in \eqref{Definition of T} that $T=\infty$. Then the fact that $T\leq T_m$ (the maximal existence time) directly concludes that $T_m=\infty$, and so \eqref{bar-u} and \eqref{convergence of u} enable us to obtain that
\begin{equation}
\begin{split}
\|u(\cdot,t)\|_{L^{\theta}}
&\leq \|u(\cdot,t)-\bar{u}_0\|_{L^{\theta}}+\|\bar{u}_0\|_{L^{\theta}}\\
&\leq 2\varepsilon c_1 e^{-\kappa t}+|\Omega|^{\frac{1}{\theta}-\frac{4}{n}}\varepsilon \leq (2 c_1 +\max\{|\Omega|^{-\frac{1}{n}},1\})\varepsilon<\infty
\end{split}
\end{equation}
for all $t\geq 1$. While, for $t<1$, the result of local existence guarantees $\|u(\cdot,t)\|_{L^{\theta}}<\infty$.
Then the uniform boundedness of $\|u(\cdot,t)\|_{L^{\theta}}$ with $\theta>\frac{n}{4}$ together with  Theorem \ref{bddness-cond} implies the boundedness of solutions $(u,v,w,z)$ as specified in the theorem.
\end{proof}
\section{Boundedness implies exponential convergence}
\subsection{Boundedness implies convergence} To our further purpose, we here observe that the discussions in \cite{FIWY16} show that the boundedness of \eqref{TI-model} indeed  implies its convergence. Indeed, by integrating the third equation in \eqref{TI-model} over $\Omega\times (0,t)$ and recalling that $w$ is bounded, cf. Lemma \ref{basic prop} (iii), we get  the following observation:
\begin{lemma}[\cite{FIWY16}] The solution of \eqref{TI-model} has the property  that
$$
\int_0^\infty\int_\Omega w(x,t)z(x,t)dxdt<\infty.
$$
\end{lemma}
Combing this crucial property with the parabolic regularity \eqref{bddness-cond}, we can conclude the uniform convergence of solutions as argued in \cite{FIWY16}.
\begin{lemma} Any \it{bounded} solution $(u,v,w,z)$ of \eqref{TI-model} will converge according to
\be\label{convergence-u-z2}
\Bigr\|(u(\cdot, t),v(\cdot, t),w(\cdot, t), z(\cdot, t))-(\bar{u}_0, \bar{v}_0+\bar{w}_0, 0, \bar{u}_0)\Bigr\|_{L^\infty(\Omega)}\rightarrow 0,
\ee
as $t\rightarrow \infty$, where $\bar{u}_0, \bar{v}_0$ and $\bar{w}_0$ are defined by \eqref{average-u0-z0}.
\end{lemma}
\subsection{Exponential convergences and their convergence rates}
In this subsection, we will show that the bounded solution of \eqref{TI-model} converges not only uniformly but also exponentially. Beyond that, by carefully collecting the appearing constants, we shall calculate out their explicit rates of convergence in terms of initial datum $u_0$ and the first nonzero Neumann eigenvalue $\lambda_1$ of $-\Delta$ in $\Omega$. Hereafter, we shall assume that the solution quadruple $(u,v,w,z)$ of \eqref{TI-model} is bounded and so we have the convergence \eqref{convergence-u-z2}. Also, the condition $u_0\geq,\not\equiv 0$ will be assumed in order to avoid saying something trivial. Thanks to \eqref{convergence-u-z2}, we henceforth fix a $t_0\geq 0$ such that
\be\label{u-z-bdd}
\frac{\bar{u}_0}{2}\leq u,  z\leq \frac{3\bar{u}_0}{2}, \quad \frac{(\bar{v}_0+\bar{w}_0)}{2}\leq v\leq \frac{3(\bar{v}_0+\bar{w}_0)}{2}   \quad  \text{on  }  \bar{\Omega}\times[t_0, \infty).
\ee
\subsubsection{Convergence rate of $w$}
\begin{lemma}\label{decay-w} The solution component $w$ converges exponentially to  $0$:
\begin{equation}\label{decay-v1}
\|w(\cdot,t)\|_{L^\infty(\Omega)}\leq\|w_0\|_{L^\infty(\Omega)}e^{-\frac{\bar{u}_0}{2}(t-t_0)}, \quad \quad \forall t\geq t_0.
\end{equation}
\end{lemma}
\begin{proof} Solving the third equation in \eqref{TI-model}, we get
$$
w(x,t)=w(x,t_0)e^{-\int_{t_0}^t z(x,s)ds},
$$
which, coupled with the boundedness of $w$ in Lemma \ref{basic prop} (iii) and the fact that $z\geq \frac{\bar{u}_0}{2}$ on $\bar{\Omega}\times[t_0, \infty)$, c.f. \eqref{u-z-bdd},   trivially  leads to  the desired estimate \eqref{decay-v1}.
\end{proof}

\subsubsection{Convergence rate of $v$} In this subsection, we will compute the exponential  convergence rate of $v$.  In fact, we obtain  the following explicit decaying estimate for $v$.
\begin{lemma}\label{cv}
The solution component $v$ has  the following decay rate: for all $t\geq t_0$,
\begin{equation}\label{cv-1}
\|v(\cdot,t)-(\bar{v}_0+\bar{w}_0)\|_{L^\infty(\Omega)}\leq \Bigr[6k_1(\bar{v}_0+\bar{w}_0)  +\left(\frac{36k_1}{e}+1\right)\|w_0\|_{L^\infty(\Omega)}\Bigr]e^{-\min\{\lambda_1,\frac{\bar{u}_0}{3}\}(t-t_0)}.
\end{equation}
\end{lemma}
\begin{proof}From the second equation in \eqref{TI-model}, we get
\begin{equation}\label{cv-2}
\left(v-\bar{v}\right)_t=\Delta(v-\bar{v})+wz-\overline{wz},
\end{equation}
by noting that
\begin{equation*}\label{cv-5}
\frac{d}{dt}\bar{v}=\overline{wz}.
\end{equation*}
The variation-of-constant formula applied to \eqref{cv-2} leads to
\begin{equation*}
v(\cdot,t)-\bar{v}=e^{(t-t_0)\Delta}\left(v(\cdot,t_0)-\bar{v}(t_0)\right)+\int_{t_0}^te^{(t-s)\Delta}\left(wz-\overline{wz}\right)ds, \quad t\geq t_0,
\end{equation*}
which, together with \eqref{Lp-3},  \eqref{u-z-bdd} and \eqref{decay-v1},  gives
\begin{equation}\label{cv-3}
\begin{split}
&\|v(\cdot,t)-\bar{v}\|_{L^\infty}\\
&\leq\|e^{(t-t_0)\Delta}\left(v(\cdot,t_0)-\bar{v}(t_0)\right)\|_{L^\infty}+\int_{t_0}^t\|e^{(t-s)\Delta}(wz-\overline{wz})\|_{L^\infty}ds\\
&\leq  2 k_1\|v(\cdot,t_0)-\bar{v}(t_0)\|_{L^\infty} e^{-\lambda_1 (t-t_0)}+2k_1\int_{t_0}^t e^{-(t-s)\lambda_1}\|(wz)(s)-\overline{wz}(s))\|_{L^\infty}ds\\
&\leq 6k_1(\bar{v}_0+\bar{w}_0) e^{-\lambda_1 (t-t_0)}+6\bar{u}_0k_1\|w_0\|_{L^\infty}\int_{t_0}^t e^{-(t-s)\lambda_1}e^{-\frac{\bar{u}_0}{2}(s-t_0)}ds\\
&\leq 6k_1(\bar{v}_0+\bar{w}_0)  e^{-\lambda_1 (t-t_0)}+\frac{36}{e}k_1\|w_0\|_{L^\infty} e^{-\frac{\bar{u}_0}{3}(t-t_0)}\\
&\leq 6k_1\Bigr[(\bar{v}_0+\bar{w}_0)  +\frac{6}{e}\|w_0\|_{L^\infty}\Bigr]e^{-\min\{\lambda_1,\frac{\bar{u}_0}{3}\}(t-t_0)},
\end{split}
\end{equation}
where we have used the following algebraic calculations:
$$
\ba{ll}
\int_{t_0}^t e^{-(t-s)\lambda_1}e^{-\frac{\bar{u}_0}{2}(s-t_0)}ds&=\begin{cases}
(t-t_0) e^{-\lambda_1(t-t_0)}, &\text{ if  } \lambda_1=\frac{\bar{u}_0}{2}\\
\frac{1}{\lambda_1-\frac{\bar{u}_0}{2}}\Bigr[e^{-\lambda_1(t-t_0)-e^{-\frac{\bar{u}_0}{2}(t-t_0)}}\Bigr], & \text{ if } \lambda_1\neq \frac{\bar{u}_0}{2}
\end{cases}\\[0.25cm]
 &\leq  (t-t_0) e^{-\frac{\bar{u}_0}{2}(t-t_0)}\leq \frac{6}{\bar{u}_0e} e^{-\frac{\bar{u}_0}{3}(t-t_0)}, \quad \quad \forall t\geq t_0.
\ea
$$
On the other hand, by the mass conservation  of $v+w$ in  Lemma \ref{basic prop} (ii)  and  \eqref{decay-v1} of Lemma \ref{decay-w}, we deduce
\begin{equation*}
\|\bar{v}-(\bar{v}_0+\bar{w}_0)\|_{L^\infty}=\|-\bar{w}\|_{L^\infty(\Omega)}\leq\|w\|_{L^\infty}\leq \|w_0\|_{L^\infty}e^{-\frac{\bar{u}_0}{2}(t-t_0)} ,
\end{equation*}
which together with \eqref{cv-3} leads to  \eqref{cv-1}.
\end{proof}
\subsubsection{Convergence rate of $u$}In this subsection, we shall illustrate that $u$ stabilizes exponentially to its spatial mean. To this end,  a decay rate of  $\nabla v$ in $L^p(\Omega)$ for arbitrary  $p>2$ is needed and is essential in our subsequent analyses.
\begin{lemma}\label{decay-v8} For any $2<p<\infty$, the gradient of $v$ has the following decay estimate:
\begin{equation}\label{decay-v9}
\|\nabla v(\cdot,t)\|_{L^p(\Omega)}\leq C e^{-\frac{1}{2}\min\{\lambda_1, \frac{\bar{u}_0}{3}\}(t-t_0)}, \quad \forall t\geq t_0,
\end{equation}
where
\be\label{C-formula}
C=C(\|\nabla v(\cdot,t_0)\|_{L^p}, \bar{u}_0,\lambda_1)= \Bigr[2k_3\|\nabla v(\cdot,t_0)\|_{L^p}+\frac{3}{2}k_2\bar{u}_0\|w_0\|_{L^\infty}|\Omega|^{\frac{1}{p}}\max\{A, B\}\Bigr]
\ee
with $A=A(\bar{u}_0,\lambda_1)$ and $B=B(\bar{u}_0)$ defined respectively by \eqref{A-formula} and \eqref{B-formula} below.
\end{lemma}
\begin{proof}Applying the variation of constants representation to the second equation  in  \eqref{TI-model},  then,  for all $t\geq t_0$, we have
\begin{equation*}
v(\cdot,t)=e^{(t-t_0)\Delta}v(\cdot,t_0)-\int_{t_0}^t e^{(t-s)\Delta}w(\cdot,s)z(\cdot,s)ds,
\end{equation*}
 which implies
\begin{equation}\label{decay-v10}
\|\nabla v(\cdot,t)\|_{L^p}\leq\|\nabla e^{(t-t_0)\Delta} v(\cdot,t_0)\|_{L^p}+\int_{t_0}^t\|\nabla e^{(t-s)\Delta } w(\cdot,s)z(\cdot,s)\|_{L^p}ds.
\end{equation}
For $p>2$, we  use the heat semigroup property  \eqref{Lp-1} to derive
\begin{equation}\label{decay-v11}
\|\nabla e^{(t-t_0)\Delta} v(\cdot,t_0)\|_{L^p}\leq 2k_3\|\nabla v(\cdot,t_0)\|_{L^p} e^{-\lambda_1(t-t_0)}, ~\forall t\geq t_0.
\end{equation}
 Moreover, using \eqref{Lp-2} and the H\"{o}lder inequality, one infers from \eqref{u-z-bdd} and \eqref{decay-v1} that
\begin{equation}\label{decay-v12}
\begin{split}
&\int_{t_0}^t\|\nabla e^{(t-s)\Delta}w(\cdot,s)z(\cdot,s)\|_{L^p}ds\\
&\leq k_2\int_{t_0}^t \left(1+(t-s)^{-\frac{1}{2}}\right)e^{-\lambda_1(t-s)}\|w(\cdot,s)z(\cdot,s)\|_{L^p}ds\\
&\leq k_2|\Omega|^{\frac{1}{p}}\int_{t_0}^t\left(1+(t-s)^{-\frac{1}{2}}\right)e^{-\lambda_1(t-s)}\|w(\cdot,s)\|_{L^\infty}\|z(\cdot,s)\|_{L^\infty}ds\\
&\leq \frac{3}{2}k_2\bar{u}_0\|w_0\|_{L^\infty}|\Omega|^{\frac{1}{p}}\int_{t_0}^t\left(1+(t-s)^{-\frac{1}{2}}\right)e^{-\lambda_1(t-s)} e^{-\frac{\bar{u}_0}{2}(s-t_0)}ds\\
&=\frac{3}{2}k_2\bar{u}_0\|w_0\|_{L^\infty}|\Omega|^{\frac{1}{p}}e^{-\frac{\bar{u}_0}{2}(t-t_0)}\int_0^{t-t_0}\left(1+\tau^{-\frac{1}{2}}\right)e^{-(\lambda_1-\frac{\bar{u}_0}{2})\tau}d\tau.
\end{split}
\end{equation}
By straightforward  computations, we have
$$
\ba{ll}
\int_0^{t-t_0}\left(1+\tau^{-\frac{1}{2}}\right)e^{-(\lambda_1-\frac{\bar{u}_0}{2})\tau}d\tau&\leq \begin{cases}
(t-t_0) +2(t-t_0)^{\frac{1}{2}} , &\text{ if  } \lambda_1\geq \frac{\bar{u}_0}{2},\\
Ae^{(\frac{\bar{u}_0}{2}-\frac{\lambda_1}{2})(t-t_0)}, & \text{ if } \lambda_1< \frac{\bar{u}_0}{2},
\end{cases}\\[0.25cm]
 &\leq \begin{cases}
B e^{\frac{\bar{u}_0}{3}(t-t_0)}, &\text{ if  } \lambda_1\geq \frac{\bar{u}_0}{2},\\
Ae^{(\frac{\bar{u}_0}{2}-\frac{\lambda_1}{2})(t-t_0)}, & \text{ if } \lambda_1< \frac{\bar{u}_0}{2},
\end{cases}
\ea
$$
with
\be\label{A-formula}
A=A(\bar{u}_0,\lambda_1)=\sup_{s> 0}\Bigr\{e^{-(\frac{\bar{u}_0}{2}-\frac{\lambda_1}{2})s}\int_0^s\left(1+\tau^{-\frac{1}{2}}\right)e^{-(\lambda_1-\frac{\bar{u}_0}{2})\tau}d\tau\Bigr\}<\infty
\ee
and
\be\label{B-formula}
 B=B(\bar{u}_0)=\sup_{s>0}\Bigr\{(s+2s^{\frac{1}{2}})e^{-\frac{\bar{u}_0}{3}s}\Bigr\}<\infty.
\ee
Then inserting these inequalities  into $\eqref{decay-v12}$ , we derive
\be\label{decay-v14}
\begin{split}
&\int_{t_0}^t\|\nabla e^{(t-s)\Delta}w(\cdot,s)z(\cdot,s)\|_{L^p}ds\\
&\leq \frac{3}{2}k_2\bar{u}_0\|w_0\|_{L^\infty}|\Omega|^{\frac{1}{p}}\max\{A, B\}e^{-\frac{1}{2}\min\{\lambda_1, \frac{\bar{u}_0}{3}\}(t-t_0)}.
\end{split}
\ee
Substituting \eqref{decay-v14} and \eqref{decay-v11} into \eqref{decay-v10}, we end up with  \eqref{decay-v9}.
\end{proof}
Now,  we are well prepared to show the rate of exponential convergence of $u$.
\begin{lemma} \label{decay-u}The $u$-component has an exponential decay estimate that
\begin{equation}\label{decay-u1}
\|u(\cdot,t)-\bar{u}_0\|_{L^\infty(\Omega)}\leq \left(5k_1+\frac{3}{2} k_4CD\right)\bar{u}_0e^{-\frac{1}{2}\min\{\lambda_1, \frac{\bar{u}_0}{3}\}(t-t_0)}, \quad \forall t\geq t_0,
\end{equation}
where $C$ and $D$ are defined by \eqref{C-formula} and \eqref{D-formula}, respectively.
\end{lemma}
\begin{proof} It follows from the first equation in  \eqref{TI-model} that
\begin{equation}\label{decay-u2}
(u-\bar{u})_t=\Delta(u-\bar{u})-\nabla\cdot(u\nabla v).
\end{equation}
The variation-of-constant formula then shows from \eqref{decay-u2} that
\begin{equation*}
u(\cdot,t)-\bar{u}=e^{(t-t_0)\Delta}(u(\cdot,t_0)-\bar{u}(t_0))-\int_{t_0}^te^{(t-s)\Delta}\nabla\cdot(u(\cdot,s)\nabla v(\cdot,s))ds,
\end{equation*}
which gives rise to
\begin{equation}\label{decay-u3}
\begin{split}
\|u(\cdot,t)-\bar{u}\|_{L^\infty}
\leq&\|e^{(t-t_0)\Delta}(u(\cdot,t_0)-\bar{u}(t_0))\|_{L^\infty}\\
&+\int_{t_0}^t\|e^{(t-s)\Delta}\nabla\cdot(u(\cdot,s)\nabla v(\cdot,s))\|_{L^\infty}ds.
\end{split}
\end{equation}
Using \eqref{Lp-3} and  noting that $\int_\Omega (u-\bar{u})=0$ and $\bar{u}=\bar{u}_0$, one has from \eqref{u-z-bdd} that
\begin{equation}\label{decay-u4}
\|e^{(t-t_0)\Delta}(u(\cdot,t_0)-\bar{u})\|_{L^\infty}\leq 2k_1\|(u(\cdot,t_0)-\bar{u}_0)\|_{L^\infty} e^{-\lambda_1(t-t_0)}\leq 6k_1\bar{u}_0e^{-\lambda_1(t-t_0)}.
\end{equation}
Furthermore,  according  to \eqref{Lp-4},  \eqref{u-z-bdd} and Lemma \ref{decay-v8} with $p=3n>2$, we deduce
\begin{equation}\label{decay-u5}
\begin{split}
&\int_{t_0}^t\|e^{(t-s)\Delta}\nabla\cdot(u(\cdot,s)\nabla v(\cdot,s))\|_{L^\infty}ds\\
&\leq k_4\int_{t_0}^t\left(1+(t-s)^{-\frac{1}{2}-\frac{n}{2*3n}}\right)e^{-\lambda_1(t-s)}\|u(\cdot,s)\|_{L^\infty}\|\nabla v(\cdot,s)\|_{L^{3n}}ds\\
&\leq\frac{3\bar{u}_0}{2} k_4C\int_{t_0}^t\left(1+(t-s)^{-\frac{2}{3}}\right)e^{-\lambda_1(t-s)}  e^{-\frac{1}{2}\min\{\lambda_1, \frac{\bar{u}_0}{3}\}(s-t_0)}ds\\
&=\frac{3\bar{u}_0}{2} k_4Ce^{-\frac{1}{2}\min\{\lambda_1, \frac{\bar{u}_0}{3}\}(t-t_0)}\int_0^{t-t_0}(1+\tau^{-\frac{2}{3}})e^{-(\lambda_1-\frac{1}{2}\min\{\lambda_1, \frac{\bar{u}_0}{3}\})\tau}d\tau\\
&\leq \frac{3\bar{u}_0}{2} k_4CDe^{-\frac{1}{2}\min\{\lambda_1, \frac{\bar{u}_0}{3}\}(t-t_0)}
\end{split}
\end{equation}where
\be\label{D-formula}
D=D(\bar{u}_0, \lambda_1)= \sup_{s>0}\int_0^s(1+\tau^{-\frac{2}{3}})e^{-(\lambda_1-\frac{1}{2}\min\{\lambda_1, \frac{\bar{u}_0}{3}\})\tau}d\tau<\infty.
\ee
 Substituting $\eqref{decay-u4}$ and $\eqref{decay-u5}$ into $\eqref{decay-u3}$,  we obtain \eqref{decay-u1}.
\end{proof}
\subsubsection{Convergence of $z$}
Finally, we compute the rate of exponential convergence of $z$.
\begin{lemma} \label{cz}The $z$-component has the following exponential decay estimate:
\begin{equation}\label{cz-1}
\|z(\cdot,t)-\bar{u}_0\|_{L^\infty(\Omega)}\leq \bar{u}_0\Bigr\{ \frac{5}{2}+2k_1\Bigr[3+\frac{ (10k_1+3k_4CD)}{2(\lambda_1+1)-\min\{\lambda_1, \frac{\bar{u}_0}{3}\}}\Bigr]\Bigr\}e^{-\min\{1, \frac{\lambda_1,}{2} \frac{\bar{u}_0}{6}\}(t-t_0)}
\end{equation}
for all $t\geq t_0$, where $C$ and $D$ are given by \eqref{C-formula} and \eqref{D-formula}.
\end{lemma}
\begin{proof}Notice from the fourth equation in \eqref{TI-model} that $(\bar{z})_t=-\bar{z}+\bar{u}=-\bar{z}+\bar{u}_0$, we have
\begin{equation}\label{cz-2}
(z-\bar{z})_t=\Delta (z-\bar{z})-(z-\bar{z})+u-\bar{u}_0.
\end{equation}
The variation-of-constant formula applied to \eqref{cz-2} gives that
\begin{equation}\label{cz-3}
z(\cdot,t)-\bar{z}= e^{(t-t_0)(\Delta-1)}(z(\cdot,t_0)-\bar{z}(t_0))+\int_{t_0}^te^{(t-s)(\Delta-1)}(u(\cdot, s)-\bar{u}_0)ds.
\end{equation}
Now, it follows from \eqref{Lp-3} and \eqref{u-z-bdd} that
\begin{equation}\label{cz-4}
\begin{split}
\|e^{(t-t_0)(\Delta-1)}(z(\cdot,t_0)-\bar{z}(t_0))\|_{L^\infty}&\leq 2k_1e^{-(\lambda_1+1)(t-t_0)}\|z(\cdot,t_0)-\bar{z}(t_0))\|_{L^\infty}\\
&\leq 6k_1\bar{u}_0 e^{-(\lambda_1+1)(t-t_0)}
\end{split}
\end{equation}
and, thanks to \eqref{decay-u1},
\begin{equation}\label{cz-5}
\begin{split}
&\Bigr\|\int_{t_0}^te^{(t-s)(\Delta-1)}(u(\cdot, s)-\bar{u}_0)ds\Bigr\|_{L^\infty}\\
&\leq 2k_1\int_{t_0}^te^{-(\lambda_1+1)(t-s)}\|u(\cdot,s)-\bar{u}_0\|_{L^\infty}ds\\
&\leq2k_1(5k_1+\frac{3}{2} k_4CD)\bar{u}_0e^{-\frac{1}{2}\min\{\lambda_1, \frac{\bar{u}_0}{3}\}(t-t_0)}\int_{0}^{t-t_0}e^{-[(\lambda_1+1)-\frac{1}{2}\min\{\lambda_1, \frac{\bar{u}_0}{3}\}]s}ds\\
&\leq  \frac{ 2k_1(10k_1+3k_4CD)\bar{u}_0}{2(\lambda_1+1)-\min\{\lambda_1, \frac{\bar{u}_0}{3}\}}e^{-\frac{1}{2}\min\{\lambda_1, \frac{\bar{u}_0}{3}\}(t-t_0)}.
\end{split}
\end{equation}
By taking $L^\infty$-norm  on both sides of \eqref{cz-3} and using \eqref{cz-4} and \eqref{cz-5}, we arrive at
\begin{equation}\label{cz-6}
\|z(\cdot,t)-\bar{z}\|_{L^\infty}\leq 2k_1\bar{u}_0\Bigr[3+\frac{ (10k_1+3k_4CD)}{2(\lambda_1+1)-\min\{\lambda_1, \frac{\bar{u}_0}{3}\}}\Bigr]e^{-\frac{1}{2}\min\{\lambda_1, \frac{\bar{u}_0}{3}\}(t-t_0)},  \forall t\geq t_0.
\end{equation}
On the other hand, we note from $(\bar{z})_t=-\bar{z}+\bar{u}$ and $\bar{u}=\bar{u}_0$  that
$$
\bar{z}(t)-\bar{u}_0=[\bar{z}(t_0)-\bar{u}_0]e^{-(t-t_0)}+e^{-t}\int_{t_0}^t(\bar{u}(s)-\bar{u}_0)e^sds=[\bar{z}(t_0)-\bar{u}_0]e^{-(t-t_0)}.
$$
Then \eqref{u-z-bdd}  implies
$$
\|\bar{z}(t)-\bar{u}_0\|_{L^\infty}\leq \|\bar{z}(t_0)-\bar{u}_0\|_{L^\infty}e^{-(t-t_0)}\leq \frac{5\bar{u}_0}{2}e^{-(t-t_0)}.
$$
This combined with \eqref{cz-6} gives the desired estimate  \eqref{cz-1}.
\end{proof}

The main exponential decay estimates in \eqref{exp-convergence-u-z} of Theorem \ref{mthm} follow a collection of Lemmas \ref{decay-w} -\ref{cz} with perhaps some large constants $m_i$. At the end of this paper, we remark that Lemma \ref{cz} may be shown alternatively via comparison principle.

\textbf{Acknowledgments}  The research of H.Y. Jin was supported by Project Funded by the NSF of China (No. 11501218), China Postdoctoral Science Foundation (No. 2015M572302) and the Fundamental Research Funds for the Central Universities (No. 2015ZM088). The research of  T. Xiang  was  funded by China Postdoctoral Science Foundation (No.2015M570190), the Fundamental Research Funds for the Central Universities, and the Research Funds of Renmin University of China (No.15XNLF10) and NSF of China (No. 11571363).


\end{document}